\documentclass[11pt,oneside,reqno]{amsart}
\usepackage{mathpazo} 
\normalfont
\usepackage[T1]{fontenc}

\usepackage[utf8]{inputenc}
\usepackage{amsmath}
\usepackage{amsfonts}
\usepackage{amssymb}
\usepackage{amsthm}
\usepackage{tikz}
\usepackage{tikz-cd}
\usepackage[all]{xy}
\usepackage[margin=1.2in]{geometry}
\usepackage{amscd}
\usepackage[shortlabels]{enumitem}

\DeclareMathOperator{\Hom}{Hom}
\DeclareMathOperator{\Spec}{Spec}

\newcommand{\angles}[1]{\left\langle #1 \right\rangle}

\input xy
\xyoption{all}
\theoremstyle{definition}
\newtheorem{mydef}{\textbf{Definition}}[section]
\newtheorem{myeg}[mydef]{\textbf{Example}}

\newtheorem{rmk}[mydef]{\textbf{Remark}}

\theoremstyle{plain}

\newtheorem*{nothma}{\textbf{Theorem A}}
\newtheorem*{nothmb}{\textbf{Theorem B}}

\newtheorem{lem}[mydef]{\textbf{Lemma}}
\newtheorem{pro}[mydef]{\textbf{Proposition}}

\newtheorem{cor}[mydef]{\textbf{Corollary}}

\newcommand{\T}{\mathbb{T}}
\newcommand{\TT}{\T}

\newcommand{\multipleaffil}[3]{%
  \address{%
    \begin{minipage}[t]{\textwidth}
      #1 \\
      #2 \\
      #3 
    \vspace{0.1cm}
    \end{minipage}
  }
}

\begin{document}

\title{Representation theory over semifields}
\author{Jaiung Jun}
\multipleaffil{Department of Mathematics, State University of New York at New Paltz, NY, USA}{and}{Institute for Advanced Study, Princeton, NJ, USA}
\email{junj@newpaltz.edu, junj@ias.edu}

\author{Kalina Mincheva}
\address{Department of Mathematics, Tulane University, New Orleans, LA 70118, USA}
\email{kmincheva@tulane.edu}

\author{Jeffrey Tolliver}
\address{}
\email{jeff.tolli@gmail.com}
\makeatletter
\@namedef{subjclassname@2020}{%
	\textup{2020} Mathematics Subject Classification}
\makeatother

\thanks{}

\begin{abstract}
We study and classify representations of a torsion group $G$ over an idempotent semifield with special attention on the case over the Boolean semifield $\mathbb{B}$. In subsequent work we extend this theory to studying representations of matroids of low rank.

\end{abstract}

\maketitle

\tableofcontents


\section{Introduction}

In this paper we lay the foundations of representation theory over semifields by developing the necessary module theoretic tools. Our main motivation comes from tropical geometry. Tropical geometry studies algebro-geometric structures via their combinatorial ``shadows'' called tropical varieties. Our work aligns with an emerging approach to study the tropical varieties using their underlying (semiring) algebra as in 
\cite{giansiracusa2016equations}, \cite{giansiracusa2018grassmann}, \cite{maclagan2016tropical}, \cite{maclagan2018tropical}, \cite{lorscheid2015scheme}, \cite{joo2014prime}, \cite{CGM2017}, \cite{LR2017systems}.

The algebraic approach can further strengthen our understanding of the related combinatorial structures. Tropical linear varieties, i.e. tropical linear spaces, can be identified with (valuated) matroids. Thus understanding tropical, and more generally semiring, representations opens the door to understanding representation theory of (valuated) matroids. 
It is natural to ask matroids (and their variations) can be seen as certain modules over semirings. 

In this paper we representations of a torsion group $G$ over an idempotent semifield. We also develop some of the technical tools to study and classify tropical (sub)representations of matroids in \cite{JMTmatroidsPart1.2} and reinterpret tropical representations via (weak) automorphism groups of valuated matroids in a companion paper \cite{JMTmatroidsPart2}. In particular, our findings suggest that by developing module theory or representation theory over semifields, one may shed some light on matroids and vice versa. 



One of the key observations in this paper is that representations of algebraic group schemes over a semifield $K$ are nothing but equivariant vector bundles on $\Spec K$ (Proposition \ref{proposition: representations as equivariant vector bundles}). This correspondence together with our recent results in \cite{jun2024equivariant} on equivariant bundles on toric varieties over semifields, lead to interesting insights about representations. For instance, the fact that any equivariant bundle equivariantly splits, proved in \cite{jun2024equivariant}, translates into the statement that any representation of an irreducible algebraic group over an idempotent semifield decomposes as a sum of one dimensional representations. 

While inspired by the relation to vector bundles, the project has evolved to study a notion of representation theory over semifields as well as the theory of modules over an idempotent semifields. These explorations meet matroid theory through tropical representation theory (and matroidal representation theory) in \cite{JMTmatroidsPart1.2} and \cite{JMTmatroidsPart2}.

\subsection{Summary of results}



Let $V$ be a free module over an idempotent semifield $K$.\footnote{What follows is true even more generally, as $K$ only needs to be a connected zero-sum-free semiring with the condition that the units of $K$, $K^\times$, form a torsion free group - this is guaranteed if $K$ is an idempotent semifield (Lemma \ref{lemma: torsionfree}).} The two most prominent examples of idempotent semifields will be the tropical semifield $\mathbb{T} = \mathbb{R} \cup \{-\infty\}$ with operations maximum and addition of real numbers, and its subfield, the Boolean semifield $\mathbb{B} = \{0, -\infty\}$. 

Fix a basis for $V$. By a \textit{basis line}, we mean a submodule of $V$ spanned by a single basis vector. By appealing to the structure of $\text{GL}_n(K)$ we show that the set of basis lines of $V$ do not depend on the choice of basis (Lemma \ref{lemma: basis lines}). Then we prove the following classification theorem for representations of a torsion group $G$.

\begin{nothma}[Propositions \ref{proposition: representations of finite groups} and \ref{proposition: double cosets}]
Let $G$ be a finite group (or more generally a torsion group) and $K$ an idempotent semifield. 
\begin{enumerate}
    \item 
There is a one-to-one correspondence between isomorphism classes of indecomposable representations of $G$ over $K$ and conjugacy classes of subgroups $H \subseteq G$.
    \item 
There is a one-to-one correspondence between isomorphism classes of representations and isomorphism classes of $G$-sets, sending a representation to its set of basis lines.  
\item 
Let $V$ and $W$ be indecomposable representations of a finite group $G$ over $K$.  Let $H_V$ and $H_W$ be subgroups corresponding to $V$ and $W$ as above. Then the set of homomorphisms $\phi: V\rightarrow W$ is in bijective correspondence with the set of all functions $H_V \backslash G / H_W \rightarrow K$.  Here $H_V \backslash G / H_W$ denotes the set of double cosets of the form $H_V g H_W$, where $g\in G$.
\end{enumerate}
In the case $K = \mathbb{B}$, (1) and (2) hold without any assumption on $G$.
\end{nothma}

Our next result is about $\mathbb{B}[G]$-modules for finite groups. One key notion that we explore is ``\emph{quasi-freeness}''. For a finitely generated $\mathbb{B}$-module, quasi-freeness is equivalent to the corresponding lattice being atomic, and it may serve as an important definition in developing linear algebra over semirings.\footnote{In our companion paper \cite{JMTmatroidsPart2}, we study the notions of ``weakly-free'' and ``quasi-free'' intensively to develop valuated matroidal representation theory. In \cite{JMTmatroidsPart2}, we also prove the equivalence between quasi-freeness and atomic lattice.} 


\begin{nothmb}[Proposition~\ref{prop:BGcyclicBqf}]
    Let $G$ be a finite group.  Let $M$ be a nonzero cyclic $\mathbb{B}[G]$-module.  Then $M$ is a quasi-free $\mathbb{B}$-module (as in Definition \ref{definition: quasi-free}).
\end{nothmb}

\bigskip

\textbf{Acknowledgment} J.J. acknowledges AMS-Simons Research Enhancement Grant for Primarily Undergraduate Institution (PUI) Faculty during the writing of this paper, and parts of this research was done during his visit to the Institute for Advanced Study supported by the Bell System Fellowship Fund. K.M. acknowledges the support of the Simons Foundation, Travel Support for Mathematicians.

\section{Preliminaries}\label{section: preliminaries}


A \textit{semiring} is a set $R$ with two binary operations (addition $+$ and multiplication $\cdot$ ) satisfying the same axioms as rings, except the existence of additive inverses. In this paper, a semiring is always assumed to be commutative. A semiring $(R,+,\cdot)$ is \emph{semifield} if $(R\backslash\{0_R\},\cdot)$ is a group. A semiring $R$ is said to be \emph{zero-sum free} if $a+b=0$ implies $a=b=0$ for all $a,b \in R$. 

We will denote by $\mathbb{B}$ the semifield with two elements $\{1,0\}$, where $1$ is the multiplicative identity, $0$ is the additive identity and $1+1 = 1$. 
The {\it tropical semifield}, denoted $\mathbb{T}$, is the set $\mathbb{R}  \cup \{-\infty\} $ with the $+$ operation to be the maximum and the $\cdot$ operation to be the usual addition, with $-\infty = 0_\mathbb{T}$. 

Modules over semirings are defined analogously to modules over rings, but there is a big difference in how they behave. 

\begin{mydef}
An \emph{affine semiring scheme} is the prime spectrum $X=\Spec R$ of some semiring $R$, equipped with a structure sheaf $\mathcal{O}_X$. A locally semiringed space is a topological space with a sheaf of semirings such that the stalk at each point has a unique maximal ideal. A \emph{semiring scheme} is a locally semiringed space which is locally isomorphic to an affine semiring scheme.
\end{mydef}

A semiring $R$ is said to be connected if $\Spec R$ is connected. This is equivalent to $R$ having only trivial idempotent pairs. See \cite[Proposition 3.10]{JMT20}. In particular, any semifield is connected.

\begin{rmk}
In \cite{giansiracusa2016equations} J. Giansiracusa and N. Giansiracusa propose a special case of the semiring schemes called tropical schemes in such as way that the tropicalization of an algebraic variety can be understood as a set of $\TT$-rational points of a tropical scheme. 
\end{rmk}

\subsection{Equivariant vector bundles over semirings}


In what follows, we will use the terms ``semiring scheme'' and ``scheme over a semiring'' interchangeably.

\begin{mydef} 
Let $X$ be a scheme over a semiring $R$. By a \emph{vector bundle} on $X$, we mean a locally free sheaf on $X$. 
\end{mydef}

We show in \cite{jun2024equivariant} that the functor of points of a locally free sheaf on a scheme $X$ over a semiring $R$ is the same as the functor of points of a geometric vector bundle on a scheme $X$ over a semiring $R$. See \cite{jun2024equivariant} for the precise definition of the functor of points of a locally free sheaf.

\begin{mydef} \label{definition: equivariant vector bundle}
Let $X$ be a scheme over a semiring $R$ and $G$ be a group scheme over $R$ acting on $X$. We define an \textit{equivariant vector bundle} to be a vector bundle $E$ on $X$ together with an action of $G(A)$ on $E(A)$ for each $R$-algebra $A$ satisfying the following:
\begin{enumerate}
    \item
the action is natural in $A$. \label{item:naturality of action on bundle}
\item 
the action makes $\pi_A$ equivariant. 
\item
the action makes the induced map $E_x \to E_{gx}$ $A$-linear.
\end{enumerate}
\end{mydef}

\begin{pro} \cite[Theorem 5.17]{jun2024equivariant} \label{proposition: split theorem}
Let $X$ be an irreducible scheme over an idempotent semifield $K$ and $G$ be an irreducible algebraic group over K acting on X. Let E be a G-equivariant vector bundle on X
which is trivial as a vector bundle. Then E is a direct sum of equivariant line bundles.   
\end{pro}

The following is a generalization of \cite[Proposition 5.6]{jun2024equivariant}. Note that in \cite[Proposition 5.6]{jun2024equivariant}, it is assumed that the group $G$ is commutative and the semifield $K$ is idempotent, however, one can easily observe that the same proof holds even when $G$ is not commutative and $K$ is zero-sum-free. 

\begin{pro}\label{proposition: K[G]-modle}
Let $K$ be a zero-sum-free semifield and $G$ be a group (not necessarily commutative), then the group semiring $K[G]$ does not have nontrivial zero-divisors. 
\end{pro}






\section{Basics of representation theory over semifields}

In this section, we classify group representations over idempotent semifields.  We consider both the case of algebraic groups and of abstract groups. When the context is clear, we will often use ``groups'' instead of ``abstract groups''.

\subsection{Representations of algebraic groups}

First recall that $\text{GL}_n$ is an algebraic group over $\mathbb{N}$ (and hence over any semiring), represented by the following semiring:
\[
A_n=\mathbb{N}[x_{ij},y_{ij}\mid 1\leq i,j, \leq n]/\sim,
\]
where $\delta_{ij}$ denotes the Kronecker delta function and $\sim$ is the congruence generated by the following:
\[
\angles{\sum_k x_{ik}y_{kj}=\delta_{ij},~~\sum_k y_{ik}x_{kj}=\delta_{ij}}
\]

We may now define representations of groups over semirings.

\begin{mydef}Let $G$ be an algebraic group over a semiring $K$.  An $n$-dimensional \emph{representation} of $G$ over $K$ is defined as a homomorphism $G\rightarrow \text{GL}_n$ of group schemes.  If instead $G$ is an abstract group, an $n$-dimensional representation of $G$ over $K$ is defined as a homomorphism $G\rightarrow \text{GL}_n(K)$ of groups.
\end{mydef}

It is clear that a group homomorphism $G\rightarrow \text{GL}_n(K)$ is equivalent to a linear action map $G \times K^n \rightarrow K^n$.  As is standard, we will often abuse notation by saying that $K^n$ is the representation.



We can also describe representations of algebraic groups as equivariant vector bundles over a point.  Note that vector bundles over $\Spec K$ for a semifield $K$ are the same as free modules.

\begin{pro}\label{proposition: representations as equivariant vector bundles}
Let $K$ be a semifield.  Let $G$ be an algebraic group over $\Spec K$.  Let $E$ be the vector bundle on $\Spec K$ corresponding to $K^n$.
\begin{enumerate}
    \item 
There is a one-to-one correspondence between natural transformations $\alpha: G\times E \rightarrow E$ which make $E$ into an equivariant vector bundle, where $G$ and $E$ are viewed as functors, and homomorphisms $G\rightarrow \text{GL}_n$.  
\item 
Furthermore, two homomorphisms $G\rightarrow \text{GL}_n$ are related by an inner automorphism of $\text{GL}_n$ if and only if the corresponding equivariant vector bundles are isomorphic.
\end{enumerate}
\end{pro}
\begin{proof}
Note that for a $K$-algebra $A$, the set $(\Spec K)(A)$ is always a one point set; we will denote its element by $p_A$.  So $E(A) = E_{p_A}$ for all $A$.  By the definition of $E$, $E_{p_A}$ may be identified with $A^n$. Hence, to give $E$ the structure of an equivariant vector bundle amounts to specifying a natural transformation
\[
\alpha_A:G(A) \times E(A) \to E(A), \quad \textrm{equivalently,} \quad \alpha_A:G(A) \times A^n \to A^n,
\]
which is a group action and is linear on the single fiber $E_{p_A} = A^n$.  Of course, linear group actions on $A^n$ correspond to homomorphisms $\chi_A:G(A) \rightarrow \text{GL}_n(A)$.  So for the first claim of the proposition, all we need to do is prove that a homomorphism
\[
\chi_A: G(A)\rightarrow \text{GL}_n(A)
\]
is natural if and only if the corresponding map 
\[
\alpha_A: G(A)\times A^n \rightarrow A^n
\]
is natural.

Let $\phi: A\rightarrow B$ be a homomorphism and $\phi_*$ be the induced maps $\Spec B \to \Spec A$, $G(A) \to G(B)$, and $\text{GL}_n(A) \to \text{GL}_n(B)$. Let $\phi^n:A^n \to B^n$ be the product of $n$-copies of $\phi$. Let $e_1,\ldots, e_n$ denote the standard basis (both in $A^n$ and $B^n$).  Let $x = \sum_i x_i e_i$ be an arbitrary element of $A^n$ and $g$ be an arbitrary element of $G(A)$. For $\alpha$ to be natural means for any $g$ and $x$,
\begin{equation}
\phi^n(\alpha_A(g, x)) = \alpha_B(\phi_*(g), \phi^n(x)). 
\end{equation}
Rewriting both the left and right side in terms of $\chi$ shows this is equivalent to (with matrix notation)
\begin{equation}
\phi^n (\chi_A(g) x) = \chi_B(\phi_*(g)) \phi^n(x)
\end{equation}
and further simplifying the left side shows this is equivalent to (with matrix notation)
\begin{equation}\label{eq: inner2}
\phi_* (\chi_A(g)) \phi^n (x) = \chi_B(\phi_*(g)) \phi^n(x).
\end{equation}
By linearity, for this to hold for all $x$ is equivalent to it holding for each $x = e_k$.  So $\alpha$ is natural if and only if for each $g$ and each $i, k$,
\begin{equation}
\phi_* (\chi_A(g)_{ik}) = \chi_B(\phi_*(g))_{ik}
\end{equation}
which is equivalent to $\chi$ being natural.

For the second claim, suppose that $\phi,\psi:G \to \text{GL}_n$ and $h \in \text{GL}_n$ such that $\phi$ and $\psi$ are related by the inner automorphism determined by $h$, i.e.
\[
\phi(g) = h \circ \psi(g) \circ h^{-1}, \forall g\in G.
\]
Let $\alpha$ and $\beta$ be the corresponding equivariant vector bundle structures on $E$, respectively. For each $K$-algebra $A$, we have $h_A\in\text{GL}_n(A)$. From the above correspondence, for $g \in G(A)$, and $v \in A^n$, we have 
\begin{equation}\label{eq: inner}
\alpha_A(g,v) = h_A\beta_A(g,h_A^{-1}v).
\end{equation}
Let $E_\alpha$ be the vector bundle $E$ with the equivariant structure $\alpha$, and likewise $E_\beta$. Consider
\[
f_A:E_\beta(A)=A^n \to E_\alpha(A)=A^n, \quad v \mapsto h_A v.
\]
We claim that $f_A$ is natural in $A$. In fact, let $\varphi:A \to B$ be a morphism of $K$-algebras. We have to show that the following diagram commutes:
\[
\begin{tikzcd}
E_\beta(A)=A^n \arrow[r,"f_A"] \arrow[d," \varphi^n",swap] & E_\alpha(A)=A^n \arrow[d,"\varphi^n"] \\
E_\beta(B)=B^n  \arrow[r,"f_B"] & E_\alpha(B)=B^n
\end{tikzcd}
\]
But, since $\varphi_*(h_A)=h_B$, we have
\[
\varphi^n(f_A(v))=\varphi^n(h_Av) = h_B\varphi^n(v) = f_B(\varphi^n(v)).
\]
Hence $f:E_\beta \to E_\alpha$ is an isomorphism of equivariant vector bundles. Conversely, let $\alpha,\beta$ be two equivariant structures on $E$. We have an isomorphism $\gamma:E_\beta \to E_\alpha$ making the following diagram commute for each $K$-algebra $A$:
\[
\begin{tikzcd}
G(A) \times E_\beta(A) \arrow[r,"\beta_A"] \arrow[d," \text{id} \times \gamma_A",swap] & E_\beta(A) \arrow[d,"\gamma_A"] \\
G(A) \times E_\alpha(A) \arrow[r,"\alpha_A"] & E_\alpha(A)
\end{tikzcd}
\]
In terms of the corresponding characters, we have the following commutative diagram:
\[
\begin{tikzcd}
G(A)  \arrow[r,"\chi_{\beta}(A)"] \arrow[d," \text{id}",swap] & \text{GL}_n(A) \arrow[d,"\eta_A"] \\
G(A)  \arrow[r,"\chi_{\alpha}(A)"] & \text{GL}_n(A) 
\end{tikzcd}
\]
Here $\eta_A$ is the inner automorphism defined via conjugation by $\gamma_A\in GL_n(A)$.  It follows that $\chi_\alpha$ and $\chi_\beta$ are related by an inner automorphism of $\text{GL}_n$. 
\end{proof}

In the case where $G$ is irreducible, we can classify representations as follows.

\begin{pro}\label{proposition: split}
Let $G$ be an irreducible algebraic group over an idempotent semifield $K$.  Then every representation of $G$ decomposes as a sum of one dimensional representations.  
\end{pro}
\begin{proof}
Observe that any vector bundle on $\Spec K$ is trivial. Hence, the claim follows immediately from Propositions \ref{proposition: split theorem} and \ref{proposition: representations as equivariant vector bundles}.
\end{proof}

At the opposite extreme we may consider an abstract group $G$.  By the same argument as in \cite[Appendix B]{jun2024equivariant}, we obtain an affine algebraic group $G_K$ over $\Spec K$ such that for any connected $K$-algebra $A$, we have an isomorphism $G \cong G_K(A)$.


Since $G_K$ is just a disjoint union of copies of $\Spec K$, a morphism of schemes $G_K \to \text{GL}_n$ is a collection of elements of $\text{GL}_n(K)$ indexed by $G$, i.e., a map 
\[
G \to \text{GL}_n(K).
\]
Then we would have to show that the further condition that $G_K \to \text{GL}_n$ is a homomorphism is equivalent to the corresponding map $G \to \text{GL}_n(K)$ being a homomorphism.

\begin{pro}\label{proposition: abstract and algebraic}
Let $G$ be a group and $K$ a semifield.  Let $G_K$ be as above.  Then there is a one-to-one correspondence between representations over $K$ of $G$ and of $G_K$.
\end{pro}
\begin{proof}
Let $E$ be a representation of $G_K$ over $K$. From Proposition \ref{proposition: representations as equivariant vector bundles}, we may assume that $E$ is a $G_K$-equivariant vector bundle on $X=\Spec K$. Then, $E$ corresponds to a finite free $K$-module $V$, and we have $E(K)=V$. On the other hand, we have $G_K(K)=G$ and hence we have an action $\varphi:G \times V \to V$, giving us a representation $\psi:G \to \text{Aut}(V)$ sending $g$ to $\varphi(g,-)$. 

Conversely, suppose that we have a representation
\begin{equation}\label{eq: representation}
\psi:G \to \text{Aut}(V)    
\end{equation}
for some finite free $K$-module $V$. Let $E$ be the vector bundle corresponding to $V$ on $X=\Spec K$. One can easily observe that $\psi$ defines a $G_K$-equivariant structure on $E$: as mentioned above, $\psi$ corresponds to a morphism of schemes $G_K \to \text{GL}_n$, which is necessarily a homomorphism since $\psi$ is a homomorphism.


\end{proof}

\subsection{Representation of abstract groups}

We now turn to the case of abstract groups rather than algebraic groups.

\begin{mydef}
Let $K$ be a semifield, and $G$ a group.  Let $V$ be a representation of $G$ over $K$.  $V$ is said to be \emph{indecomposable} if it cannot be written as a direct sum of nontrivial subrepresentations.
\end{mydef}

For the purpose of the following example, we remark that representations are equivalent to $K[G]$-modules by the classical argument, so we can define the regular representation by viewing $K[G]$ as a module over itself.  In the classical setting, the regular representation contains all irreducible representations as summands.  The zero-sum-free case behaves very differently though.

\begin{myeg}
Let $K$ be a zero-sum-free semifield, $G$ be a group, and let $V = K[G]$ be the regular representation.  Suppose we can decompose $V$ as a direct sum $V_1 \oplus V_2$ of nontrivial subrepresentations.  We obtain $K[G]$-module endomorphisms $p_1, p_2$ defined for $v_1 \in V_1$ and $v_2 \in V_2$ by $p_1(v_1 + v_2) = v_1$ and $p_2(v_1 + v_2) = v_2$.  Clearly neither is zero, but $p_1 p_2 = 0$.  This gives rise to nontrivial zero-divisors in $\mathrm{End}_{K[G]}(V) = K[G]$, in contradiction to Proposition \ref{proposition: K[G]-modle}. Thus the regular representation is indecomposable.
\end{myeg}

The key tool for studying representations of groups over idempotent semifields is the set of basis lines defined as follows, together with the group action on this set.
\begin{mydef}Let $V$ be a free module over a connected zero-sum-free semiring.  Pick a basis for $V$.  We refer to the (one-dimensional) submodules spanned by a single basis vector as the \textit{basis lines}.
\end{mydef}

\begin{lem}\label{lemma: basis lines}
Let $V$ be a free module over a connected zero-sum-free semiring.  The set of basis lines does not depend on the choice of basis.  Moreover, any automorphism $f: V \rightarrow V$ induces a permutation $\sigma_f$ of the set of basis lines by sending the line spanned by a basis vector $v$ to $\mathrm{span}(f(v))$.  Given two automorphisms $f, g$, we have $\sigma_{fg} = \sigma_f \sigma_g$.
\end{lem}
\begin{proof}We recall \cite[Proposition 3.15]{JMT20}, which says the basis of $V$ is unique up to permutation and rescaling.  A basis line is unchanged by rescaling the corresponding basis vector.  And the set of basis lines is unordered, so it is not changed by permutations of the basis either.

Given a basis $v_1, \ldots, v_n$ and an automorphism $f$, we obtain a new basis $f(v_1), \ldots, f(v_n)$.  Since this yields the same set of basis lines, the map sending $\mathrm{span}(v_i)$ to $\mathrm{span}(f(v_i))$ indeed sends basis lines to basis lines.  This map is invertible because $f$ is invertible.  

The final claim follows from the fact that both maps send $\mathrm{span}(v_i)$ to $\mathrm{span}(f(g(v_i)))$.
\end{proof}

If $V$ is a representation of $G$, we have a homomorphism from $G$ to the module automorphism group of $V$, and the above result gives a homomorphism from this automorphism group to the set of permutations of the basis lines.  Thus we have the following.

\begin{cor}
Let $K$ be a zero-sum-free semifield, and $G$ a group.  Let $V$ be a representation of $G$ over $K$.  Then $G$ acts on the set of basis lines of $V$.  Specifically, for $g \in G$ and for a vector $v$ belonging to some basis, we have $g(\mathrm{span}(v)) = \mathrm{span}(gv)$.
\end{cor}

We will show that indecomposability is equivalent to the action on the set of basis lines being transitive.  This may be used to decompose a representation as a sum of indecomposable representations by decomposing the set of basis lines into orbits.

\begin{pro}
Let $V$ be a representation of a group $G$ over a zero-sum-free semifield $K$.  Then $V$ is indecomposable if and only if $G$ acts transitively on the set of basis lines of $V$.  In addition, any representation $V$ may be uniquely decomposed as a direct sum of indecomposable representations.
\end{pro}
\begin{proof}Let $S$ be the $G$-set of basis lines.  First we relate direct sum decompositions of $V$ to union decompositions of $S$.  If $V = V_1 \oplus V_2$ is a direct sum of nontrivial representations, let $S_1$ and $S_2$ be the $G$-sets of basis lines of $V_1$ and $V_2$.  We may form a basis of $V$ as a union of bases of $V_1$ and $V_2$; by looking at the corresponding basis lines we see $S = S_1 \cup S_2$, so $S$ is a union of nonempty $G$-sets.

Conversely suppose we have a nontrivial disjoint union decomposition $S = S_1 \cup S_2$ for $G$-sets $S_1$ and $S_2$.  Then a basis for $V$ can be decomposed into basis vectors belonging to some line in $S_1$ and those belonging to a line in $S_2$.  Let $V_1$ and $V_2$ be the free modules spanned by these two subsets of the basis; thus we obtain a nontrivial decomposition $V = V_1 \oplus V_2$.  We show these summands are subrepresentations.  If $v$ is a basis vector of $V_1$, then $gv$ belongs to the line $\mathrm{span}(gv) = g(\mathrm{span}(v))$, and this is in $S_1$ (because $S_1$ is closed under the action).  Thus $gv \in V_1$.  More generally, if $v \in V_1$, write $v = \sum a_i v_i$ where $v_i$ ranges over a basis for $V_1$ and then $gv = \sum a_i (gv_i) \in V_1$.  Similarly $V_2$ is also a subrepresentation.

If $V$ is indecomposable, then there is no nontrivial direct sum decomposition, so $S$ has no nontrivial disjoint union decomposition.  But $S$ can be decomposed as a union of orbits, so there is only one orbit.  Conversely suppose $G$ acts transitively on $S$.  Then $S$ has no nontrivial disjoint union decomposition, so $V$ is indecomposable.

In general, we may write $S$ as a union $S_1 \cup \ldots \cup S_k$ of orbits, yielding a direct sum decomposition $V = V_1 \oplus \ldots \oplus V_k$.  Each $V_k$ has a set $S_k$ of basis lines with a single orbit, so is indecomposable.  As an alternative proof, we could also decompose $V$ into indecomposable representations via straightforward induction on dimension.
\end{proof}

\begin{rmk}\label{remark: subtractive irreducibility}
It is easy to see that if we define irreducibility in the obvious way, then any nonzero irreducible representation over a zero-sum-free semifield is one-dimensional.  As an alternative, we may call a representation $V$ \emph{subtractively irreducible} if every subtractive $K$-submodule\footnote{See the paragraph right before Lemma \ref{lemma: dual lemma} for the definition.} which is closed under the $G$-action is $V$ or $0$.  Over an idempotent semifield, subtractive submodules of free modules correspond to subsets of the set of basis lines.  So $V$ is subtractively irreducible over $K[G]$ if and only if the set of basis lines has no nontrivial $G$-subsets, i.e., if and only if it consists of a single orbit.  So in the idempotent case, the above proposition provides a unique decomposition as a sum of subtractively irreducible subrepresentations.
\end{rmk}

While we focus only on finite-dimensional representations, we note that the above proof applies in the infinite-dimensional case as well.

\begin{rmk}
Suppose $G$ is a finite group.  Then for any indecomposable representation over a zero-sum-free semifield, the set of basis lines is a transitive $G$-set, so its cardinality divides $|G|$.  But the cardinality of the set of basis lines is the dimension.  Thus every indecomposable representation has dimension dividing $|G|$.
\end{rmk}

 Next we classify indecomposable representations.  They correspond to pairs of a subgroup and a equivalence class of characters on the subgroup.

 In the following, for each subgroup $H$ pick a subset $S_H \subseteq G$ consisting of one element of each left coset (the one-to-one correspondences below depend on these choices, so are not quite canonical).

\begin{pro}\label{proposition: classification of indecomposable representations}
Let $G$ be a group and $K$ be a zero-sum-free semifield.  
\begin{enumerate}
\item 
There is a one-to-one correspondence between isomorphism classes of pairs\footnote{By this, we mean an equivalence class of pairs with $v \in V$ where $(V, v)$ is equivalent to $(W, w)$ if there's an isomorphism $V \to W$ mapping $v$ to $w$.} consisting of an indecomposable representation $V$ together with a vector $v$ belonging to some basis of $V$ and pairs consisting of a subgroup $H\subseteq G$ and a homomorphism  $\chi: H\rightarrow K^\times$.
\item 
There is a one-to-one correspondence between isomorphism classes of indecomposable representations and equivalence classes of pairs consisting of a subgroup $H\subseteq G$ and a homomorphism $\chi: H\rightarrow K^\times$ where $(H, \chi) \sim (H', \chi')$ if there is some $g\in G$ such that $H' = gHg^{-1}$ and for all $h'\in H$ we have $\chi'(h') = \chi(g^{-1}h'g)$.
\end{enumerate}
\end{pro}
\begin{proof}Let $V$ be an indecomposable representation and $v$ be an element of some basis (i.e. $v$ spans some basis line).  We construct a subgroup $H$ as the stabilizer of the basis line $\mathrm{span}(v)$.  For each $h \in H$, $hv$ belongs to a basis, and spans the same basis line.  Thus there is a unit $\chi(h)$ such that $hv = \chi(h)v$.  If $h'\in H$ then $\chi(hh')v = hh' v = \chi(h) \chi(h') v$.  Since $v$ is a basis vector, $\chi(hh') = \chi(h) \chi(h')$ so $\chi$ is a homomorphism.

The transitivity of the action on the set of basis lines implies each basis line has the form $\mathrm{span}(gv)$ for some $g\in G$.  Writing $g = sh$ with $s \in S_H$ and $h \in H$, each basis line has the form $\mathrm{span}(sv)$.  So we obtain a basis for $V$ of the form $\{ sv \mid s \in S_H \}$.

Given $g = sh$ for $s \in S_H$ and $h \in H$, we obtain $gv = shv = \chi(h) sv$, so $H$ and $\chi$ determine how $G$ acts on $v$.  This determines the $G$-action on all basis vectors (hence on all vectors by linearity) because acting on $s'v$ for some $s' \in S_H$ by $g$ yields $(gs')v$, and we have specified how group elements act on $v$.  Thus the pair $(V, v)$ is determined up to isomorphism by $(H, \chi)$.  

Similarly one may follow this construction to show each pair $(H, \chi)$ yields an indecomposable representation $V$, and an element $v$ of a basis. To be precise, we take $V$ to be the free module whose basis is the set $G/H$ of left cosets, and let $v$ be the coset $H$.  Every element of $G$ can be uniquely written as $sh$ where $s \in S_H$ (the set of chosen representatives of left cosets) and $h \in H$.  For a basis vector (i.e. left coset) $tH$ with $t \in S_H$, we define the action by 
\[
g (tH) = s \chi(h) v,
\]
where $gt = sh$ with $s \in S_H$ and $h \in H$. One can easily check that this induces a transitive $G$-action on the set of basis lines, in particular, it defines an indecomposable representation of $G$. 

It remains to show that two pairs $(H, \chi)$ and $(H', \chi')$ yield isomorphic representations (with different basis vectors) if and only if they are conjugate by some $g\in G$.  Let $(V, v)$ be as above and let $w$ be an element of some other basis.  Because $G$ acts transitively on basis lines, there is a unit $u\in K^\times$ and some $g\in G$ such that $w = u(gv)$.  Observe that 
\begin{equation}
u(h'gv) = h'w = \chi'(h') w = u\chi'(h') gv.
\end{equation}
We may cancel the $u$ and act by $g^{-1}$ to obtain $g^{-1}h'gv = \chi'(h') v$. Thus $\mathrm{span}(v)$ is stabilized by $g^{-1}h'g$ (so it belongs to $H$) and $\chi'(h') = \chi(g^{-1} h' g)$.  We have seen $g^{-1} H' g \subseteq H$, i.e., $H' \subseteq gHg^{-1}$.  The reverse inclusion is obtained by swapping the roles of $v$ and $w$.
\end{proof}

Things get even simpler if we focus on finite groups.  For this, we need the following well-known result from the theory of lattice-ordered groups.

\begin{lem}\label{lemma: torsionfree}
Let $K$ be an idempotent semifield.  Then $K^\times$ is torsion-free.
\end{lem}
\begin{proof}
Let $x$ be a torsion element, i.e., $x^n = 1$ for some $n$.  Let $u = 1 + x + \ldots + x^{n-1}$.  By idempotence $u + 1 = u$; since $0$ does not satisfy this equation, $u$ is nonzero and hence a unit.  Now observe that
\[ xu = x + x^2 + \ldots + x^n = x + x^2 + \ldots + x^{n-1} + 1 = 1 + x + \ldots + x^{n-1} = u \]
Dividing by $u$ gives $x = 1$ as desired.
\end{proof}

\begin{pro}\label{proposition: representations of finite groups}
Let $G$ be a finite group (or more generally a torsion group) and $K$ an idempotent semifield. 
\begin{enumerate}
    \item 
There is a one-to-one correspondence between isomorphism classes of indecomposable representations of $G$ over $K$ and conjugacy classes of subgroups $H \subseteq G$.
    \item 
Furthermore, there is a one-to-one correspondence between isomorphism classes of representations and isomorphism classes of $G$-sets, sending a representation to its set of basis lines.  
\end{enumerate}
In the case $K = \mathbb{B}$, the above hold without any assumption on $G$.
\end{pro}
\begin{proof}
(1): If $G$ is torsion, then so is any subgroup $H$.  Since $K^\times$ is torsion-free, any homomorphism $\chi: H \rightarrow K^\times$ is trivial.  If $K = \mathbb{B}$, $K^\times$ is the trivial group, so $\chi$ is trivial without any assumption on $G$. Thus we may omit the homomorphism involved in the classification from the second part of Proposition \ref{proposition: classification of indecomposable representations}.  This immediately yields the first claim.

(2): From the proof of Proposition \ref{proposition: classification of indecomposable representations}, we see that the subgroup $H$ involved in the classification is constructed as the stabilizer of some basis line, and of course the conjugacy class of $H$ is independent of the basis line chosen.  In particular, it is determined entirely by the $G$-set of basis lines.  Combining with the first part of this proposition, we see that an indecomposable representation is determined by the $G$-set of basis lines.  So we get a one-to-one correspondence between isomorphism classes of indecomposable representations and isomorphism classes of transitive $G$-sets.

Decomposing into orbits yields a one-to-one correspondence between isomorphism classes of $G$-sets and collections of isomorphism classes of transitive $G$-sets.  Similarly isomorphism classes of representations correspond to collections of isomorphism classes of indecomposable representations.  So we get a one-to-one correspondence between isomorphism classes of representations and isomorphism classes of $G$-sets by decomposing a representation into indecomposables, taking the $G$-set of basis lines of each indecomposable representation, and then taking the disjoint union.  But the $G$-set of basis lines of a direct sum is the union of $G$-sets of basis lines of each term, so this construction is the same as simply taking the $G$-set of basis lines of the original representation.
\end{proof}

Elements of an indecomposable representation of a finite group have a particularly simple description. For a group $G$ and its subgroup $H$, we let $G/H$ be the $G$-set of left cosets. 

\begin{lem}\label{lemma: indecomposable representations are subrepresentations of the regular representation}
Let $V$ be an indecomposable representation of a finite group $G$ over an idempotent semifield $K$.  Let $H$ be a subgroup in the equivalence class corresponding to $V$ under Proposition~\ref{proposition: representations of finite groups}.  Pick a section $s: G/H \rightarrow G$ of the quotient map.  Then there is some $v\in V$ such that every element of $V$ can be represented uniquely as
\[
\sum_{g\in s(G / H)} a_g gv,
\]
for some $a_g \in K$. Moreover, every element can be represented uniquely as
\[
\sum_{g\in G} a_g gv,
\]
where $a_g$ is constant on left cosets (i.e. $a_{gh} = a_g$ for all $h\in H$).
\end{lem}
\begin{proof}
By the proof of Proposition \ref{proposition: classification of indecomposable representations}, $H$ is the stabilizer of some basis line $\mathrm{span}(v)$.  Moreover, the proof of Proposition \ref{proposition: representations of finite groups} shows that $H$ acts trivially on $v$.  

Because $G$ acts transitively on the set of basis lines and $H$ stabilizes $\mathrm{span}(v)$, each basis line has the form $\mathrm{span}(gv)$ for a unique $g\in s(G/H)$.  Thus we obtain a basis for V of the form $\{ gv \mid g \in s(G/H) \}$.  

Now we can write any $x\in V$ as
\[ 
x = \sum_{g \in s(G / H)}a_g gv, 
\]
where the coefficients $a_g$ are defined for all $g \in s(G / H)$.  We extend the definition of $a_g$ to all $g \in G$ by defining $a_g$ to be constant on cosets.  We obtain
\[ 
x = \sum_{h \in H} x = \sum_{h \in H} \sum_{g \in s(G / H)} a_g gv = \sum_{h\in H, g \in s(G / H)} a_{gh} ghv = \sum_{g\in G} a_g gv.
\]
The first equality above is idempotence, the second is the definition of $a_g$, the third follows from $a_g$ being constant on cosets and $H$ acting trivially on $v$, and the fourth is because each element of $G$ can be written uniquely as the representative of its coset times some element of $H$.

That this representation is unique follows by using the same chain of equalities to write $\sum_{g\in G} a_g gv$ as $\sum_{g\in s(G/H)} a_g gv$ and using the uniqueness of this latter decomposition.
\end{proof}

This shows each indecomposable representation is a subrepresentation (but not a summand) of the regular representation, consisting of those elements $\sum a_g g$ for which $a_g$ is constant on cosets.  It is also a quotient of the regular representation.

\begin{lem}\label{lemma: indecomposable representations are quotients of the regular representation}
Let $V$ be an indecomposable representation of a finite group $G$ over an idempotent semifield $K$.  Let $K[G]$ be the regular representation.  Then there is a subgroup $H\subseteq G$ such that $V$ is isomorphic to the quotient of $K[G]$ in which $g$ is identified with $gh$ for all $g\in G$ and $h \in H$.
\end{lem}
\begin{proof}
Let $H$ be a subgroup in the equivalence class corresponding to $V$ under Proposition~\ref{proposition: representations of finite groups}. Let $W$ be the quotient of $K[G]$ by $g \sim gh$ for $g\in G$, $h\in H$.  Observe that the quotient of a free module by identifying some basis elements with other basis elements is free on the equivalence classes of the basis elements modulo $H$.  Thus $W$ is free with basis $G / H$.  Let $s: G / H \rightarrow G$ be a section of the quotient map.  Then every element of $W$ may be written uniquely as $\sum_{g\in S} a_g gH$.  The result now follows from Lemma \ref{lemma: indecomposable representations are subrepresentations of the regular representation}.
\end{proof}

We now turn to the problem of classifying homomorphisms between representations.
\begin{mydef}Let $G$ be a group and $K$ be a semifield.  Let $V$ and $W$ be representations of $G$ over $K$.  A \emph{homomorphism of representations} $\phi: V \rightarrow W$ is a module homomorphism such that $\phi(gv) = g\phi(v)$ for all $g\in G$ and $v \in V$.
\end{mydef}

If we write $V$ and $W$ as direct sums of indecomposable representations, then specifying a homomorphism $V\rightarrow W$ is the same as specifying a collection of homomorphisms consisting of one from each summand of $V$ to each summand of $W$.  Thus we may focus on classifying homomorphisms between indecomposable representations.  It is worth pointing out that Schur's lemma cannot be used for this task.  One issue is that indecomposable representations may not be irreducible in the idempotent setting.  Another issue is that a homomorphism with trivial kernel does not have to be injective.

The key insight for this classification is that if we view $V$ as a quotient of the regular representation and $W$ as a subrepresentation of the regular representation, then homomorphisms $V \rightarrow W$ are essentially endomorphisms of the regular representation that factor through the quotient $V$ and land in the subrepresentation $W$.  The next lemma shows more explicitly how to use the description of $V$ as a quotient.

\begin{lem}
Let $G$ be a finite group and let $K$ be an idempotent semifield.  Let $V$ be an indecomposable representation of $G$ over $K$ and let $W$ be any representation.  Let $H\subseteq G$ be a subgroup corresponding to $V$ under Proposition~\ref{proposition: representations of finite groups}.  Then there is a one-to-one correspondence between homomorphisms $V\rightarrow W$ and $H$-invariant elements of $W$.
\end{lem}
\begin{proof}
Clearly $V\equiv K[G]/\sim$, where $\sim$ is the module congruence on $K[G]$ generated by $g\sim gh$ for $g\in G, h\in H$. By the universal property of quotients, homomorphisms $V \rightarrow W$ correspond to homomorphisms $\phi: K[G] \rightarrow W$ satisfying $\phi(gh) = \phi(g)$.  This latter equation is equivalent to $\phi(h) = \phi(1)$ since $\phi$ preserves multiplication by $g^{-1}$.  The universal property of the regular representation (which is the free $K[G]$-module of rank $1$) says that homomorphisms $\phi: K[G]\rightarrow W$ are in one-to-one correspondence with elements $\phi(1)\in W$.  Thus homomorphisms satisfying $\phi(h) = \phi(1)$ (or equivalently $h\phi(1) = \phi(1)$) for all $h\in H$ correspond to $H$-invariant elements of $W$.
\end{proof}

We now combine this with the description of an indecomposable representation $W$ as a subrepresentation of the regular representation to obtain a classification of homomorphisms of indecomposable representations.

\begin{pro}\label{proposition: double cosets}
Let $V$ and $W$ be indecomposable representations of a finite group $G$ over an idempotent semifield $K$.  Let $H_V$ and $H_W$ be subgroups corresponding to $V$ and $W$ under Proposition~\ref{proposition: representations of finite groups}.  Then the set of homomorphisms $\phi: V\rightarrow W$ is in bijective correspondence with the set of all functions $H_V \backslash G / H_W \rightarrow K$.  Here $H_V \backslash G / H_W$ denotes the set of double cosets of the form $H_V g H_W$, where $g\in G$.
\end{pro}
\begin{proof}
Homomorphisms $V\rightarrow W$ correspond to $H_V$-invariant elements of $W$.  There is some $w\in W$ such that each element of $W$ can be written uniquely in the form $\sum\limits_{g\in G} a_g gw$ with $a_g = a_{gk}$ for each $k\in H_W$.  This element is $H_V$-invariant if and only if for all $h\in H_V$,
\[ 
\sum_{g\in G} a_g gw = \sum_{g\in G} a_g hgw = \sum_{g\in G} a_{h^{-1}g} gw. 
\]
Note that the collection of coefficients $a_{h^{-1}g}$ is constant on cosets of $H_W$.  And the above equation holds if and only if for all $h \in H_V$ and all $g \in G$, $a_g = a_{h^{-1}g}$, or equivalently if and only if $a_g = a_{hg}$ for all $h \in H_V$ and $g\in G$.

Thus $H_V$-invariant elements of $W$ have a unique description of the form $\sum\limits_{g\in G} a_g gw$ where the coefficients are right-invariant under $H_W$ and left invariant under $H_V$.  We may identify the set of homomorphisms $V\rightarrow W$ with the set of such elements, which we may in turn identify with maps $H_W \backslash G / H_V \rightarrow K$.
\end{proof}


\section{Group actions on $\mathbb{B}$-modules}\label{section: B-modules}
A finitely generated $\mathbb{B}$-module is in fact finite, and furthermore a lattice.  One can define duals of such modules, which can either be viewed as an instance of lattice duality, or as an analogue of duality for vector spaces.  Before delving into the special case of $\mathbb{B}$-modules, we consider this situation in slightly more generality.

\begin{mydef}
Let $K$ be a semiring and $M$ be a $K$-module.  $M^\vee:=\Hom_K(M, K)$ is called the dual of $M$.  $M$ is called \emph{reflexive} if the canonical map $M \rightarrow M^{\vee\vee}=\Hom_K(M^\vee, K)$ is an isomorphism, and \emph{weakly reflexive} if this map is a monomorphism.
\end{mydef}

\begin{rmk}
Reflexive modules play an important role in the classical theory, but also it provides an interesting interpretation for certain classical objects. For instance, see \cite[Theorem 13.1]{borger2024facets} 
for an application of reflexive modules over semifields to narrow class groups. 
\end{rmk}

Recall that a submodule $M$ of a $\mathbb{B}$-module $N$ is said to be \textit{subtractive} if $x,y \in N$ such that $x \leq y$ and $y \in M$, then $x \in M$, where $x\leq y$ if and only if $x+y=y$. 

\begin{lem}\label{lemma: dual lemma}
Let $M$ be a finitely generated $\mathbb{B}$-module.  Then there is an order reversing bijection $M\rightarrow M^\vee$.  Furthermore, $M$ is reflexive.
\end{lem}
\begin{proof}
One can easily check that an element of $\mathrm{Hom}(M, \mathbb{B})$ is determined by its kernel which is a subtractive submodule. 
\[
\{y\in M\mid y\leq x\}, 
\]
where $x$ is the sum of elements in the kernel.  In this way we obtain a bijection 
\[
\psi: M \rightarrow \Hom(M, \mathbb{B}), \quad m \mapsto \{y \in M \mid y \leq m\}
\]
which is not natural as the right side is a contravariant functor.

Now, observe $\Hom(M, \mathbb{B})$ is isomorphic to the dual lattice, since $\psi$ is an order reversing bijection.  Then the final claim is essentially that the dual of the dual lattice isomorphic to $M$.

\end{proof}

\begin{rmk}
Note that it is also easy to prove Lemma \ref{lemma: dual lemma} directly without using lattice theory. For this, note that both $M$ and its double dual have the same cardinality, so it suffices to prove injectivity.  If $x$ and $y$ map to the same element of the double dual, this means that $f(x) = f(y)$ for all $f: M\rightarrow\mathbb{B}$.  By taking $f = \psi(z)$ for some $z$, we obtain $\psi(z)(x) = \psi(z)(y)$ which means that $x\leq z$ if and only if $y \leq z$.  Since this holds for all $z$, $x = y$, establishing the result.
\end{rmk}


The above result relating $K$-modules and their duals is well-known, when $K$ is a field. Now we look closer into weakly reflexive modules. We show that submodules of free modules are weakly reflexive.

\begin{lem}
Let $K$ be a semiring, and let $M$ be a $K$-module.  Suppose there is a monomorphism from $M$ to a (possibly infinitely generated) free $K$-module.  Then $M$ is weakly reflexive.
\end{lem}
\begin{proof}
Given a pair of distinct elements $x, y\in M$, we may view them as distinct vectors inside a free module, and hence some coordinates differs. 
By considering the projection map onto this coordinate, we obtain a homomorphism $\phi: M \rightarrow K$ such that $\phi(x) \neq \phi(y)$.  The canonical map $M\rightarrow M^{\vee\vee}$ sends $x$ to the map $\hat{x}$ given by $\hat{x}(\psi) = \psi(x)$ and similarly for $y$.  Clearly $\hat{x}(\phi) = \phi(x) \neq \hat{y}(\phi)$ so $\hat{x} \neq \hat{y}$, which establishes weak reflexivity.
\end{proof}


\begin{rmk}
Another situation where dualizability occurs is in the case of cones (viewed as $\mathbb{R}_{\geq 0}$-modules). For interesting discussion on this in relation to flat modules and projective modules in the semiring setting, see \cite{borger2024facets}.
\end{rmk}


If a $K$-module $M$ is equipped with a left $G$-action, then $M^\vee$ carries a right $G$-action defined in the usual way.  Note that the double dual $M^{\vee\vee}$ then has a left $G$-action.  Moreover, the usual argument shows that the map $M \rightarrow M^{\vee\vee}$ preserves the action, so is a $K[G]$-module homomorphism.

\begin{lem}\label{lemma: swap sub and quotients}
Let $K$ be a semiring and $M$ be a module over $K$. If $L$ is a quotient module of $M$, then $L^\vee$ is a submodule of $M^\vee$.
\end{lem}
\begin{proof}
Let $L=M/\sim$ for some congruence relation $\sim$ on $M$. As in the case for rings, one has the following identification:
\[
\Hom(L,K) = \{ \phi \in \Hom(M,K) \mid \phi(a) = \phi(b) ~\forall a \sim b.\} 
\]
\end{proof}

In the classical theory of representations of finite groups over fields, any representation can be viewed as either a subrepresentation or a quotient representation of a direct sum of copies of the regular representation.  In the more general setting where $K$ is a semifield, writing a $K[G]$-module as a quotient of a sum of copies of the regular representation is still trivial -- it amounts to writing a $K[G]$-module in terms of generators and relations.  To embed a $K[G]$-module into a free $K[G]$-module, we will need to use the theory of dual modules.

\begin{pro}\label{proposition: sub proposition}
Let $K$ be a semiring. Let $G$ be a finite group and let $M$ be a $K[G]$-module.  Suppose $M$ is weakly reflexive as a $K$-module and $M^\vee$ is finitely generated.  Then $M$ is isomorphic to a submodule of $K[G]^n$ for some $n$.
\end{pro}
\begin{proof}
From Lemma \ref{lemma: swap sub and quotients}, it is enough to show that $M^\vee$ is a quotient of $K[G^{\text{op}}]^n$. But, since $M^\vee$ is finitely generated over $K$ and $G$ is finite, $M^\vee$ is also finitely generated over $K[G^{\text{op}}]$. In particular, $M^\vee$ is a quotient of $K[G^{\text{op}}]^n$ for some $n$. 
\end{proof}

\begin{cor}
Let $G$ be a finite group and $M$ be a finitely generated (hence finite) $\mathbb{B}[G]$-module.  Then $M$ is isomorphic to a submodule of $\mathbb{B}[G]^n$ for some $n$.
\end{cor}
\begin{proof}
As we note right before Lemma \ref{lemma: swap sub and quotients}, $\Hom(M,\mathbb{B})$ is naturally equipped with a right $G$-action, and hence it is a $\mathbb{B}[G^\text{op}]$-module. Likewise, $\Hom(\Hom(M,\mathbb{B}),\mathbb{B})$ is a $\mathbb{B}[G]$-module. 

Since $\Hom(M,\mathbb{B})$ is finite, it is a finitely generated $\mathbb{B}[G^\text{op}]$-module, which makes it a quotient of $\mathbb{B}[G^\text{op}]^n$. Therefore, by dualizing $\Hom(M,\mathbb{B})$ (as a $\mathbb{B}$-module), we get that $M$ is a $\mathbb{B}$-submodule of $\mathbb{B}[G]^n$ by Lemmas \ref{lemma: dual lemma} and \ref{lemma: swap sub and quotients}. So, we have an inclusion of $\mathbb{B}$-modules:
\[
\iota:M \to \mathbb{B}[G]^n.
\]
But since the inclusion map $\iota$ is obtained by dualizing a map of $\mathbb{B}[G^\text{op}]$-module, it preserves the left $G$-action. In particular, $\iota$ is a $\mathbb{B}[G]$-module map, showing that $M$ is isomorphic to a submodule of $\mathbb{B}[G]^n$.
\end{proof}



In general $\mathbb{B}[G]$-modules can be quite complicated.  So it is natural to study special cases such as cyclic $\mathbb{B}[G]$-modules (i.e. those generated by a single element).  One observation in this vein is that the underlying $\mathbb{B}$-module must be ``quasi-free''. 

\begin{mydef}\label{definition: quasi-free}
Let $M$ be a module over a semiring $R$ and let $x_1,\ldots, x_n\in M$.
\begin{enumerate}
    \item 
Elements $x_1, \ldots, x_n$ are said to be \emph{quasi-independent} if an equation of the form $x_i = \sum_j c_j x_j$ implies $c_j = \delta_{ij}$.
    \item 
A \emph{quasi-basis} is a quasi-independent set of generators. 
\item 
$M$ is said to be \emph{quasi-free} of rank $n$ if it has a quasi-basis with $n$ elements. 
\end{enumerate}
Quasi-free will always mean quasi-free of finite rank unless otherwise specified.
\end{mydef}

\begin{lem}\label{lemma: anti-chain lemma}
Let $G$ be a torsion group.  Let $M$ be a finitely generated $\mathbb{B}[G]$-module.  Then each orbit of the $G$-action on the poset of join-irreducible elements forms an antichain. 
\end{lem}
\begin{proof}Let $v\in M$ be join-irreducible.  If $v\leq gv$ for some $g\in G$, choose $n$ such that $g^n = 1$.  By repeatedly applying $g$ to both sides of $v\leq gv$, we obtain
\begin{equation}
v \leq gv \leq g^2 v \leq \ldots \leq g^n v = v
\end{equation}
so $v = gv$.  

If $gv \leq hv$ for some $g, h\in G$, then $v\leq g^{-1}hv$.  By the above $v = g^{-1}hv$ so $gv = hv$.  Thus the orbit of $v$ is an antichain.
\end{proof}

\begin{pro}\label{prop:BGcyclicBqf}
Let $G$ be a finite group.  Let $M$ be a nonzero cyclic $\mathbb{B}[G]$-module.  Then $M$ is a quasi-free $\mathbb{B}$-module.
\end{pro}
\begin{proof}
Let $v\in M$ generate $M$ as a $\mathbb{B}[G]$-module.  Let $H\leq G$ be the stabilizer of $v$.  Since $gv=ghv$  for any $h \in H$, we may write it as $(gH)v$.  Clearly $M$ is generated as a $\mathbb{B}$-module by $\{ (gH)v \mid gH\in G/H \}$. Also, we remark that from Lemma \ref{lemma: anti-chain lemma} if $gv \leq v$ for some $g\in G$ then $g\in H$.  


Suppose we have a relation of the form
\[ 
(gH)v = \sum_{kH\in G/H} a_{kH} (kH)v. 
\]
If $a_{kH} \neq 0$, we have $kv = (kH)v \leq gv$.  Multiplying by $g^{-1}$, we get $g^{-1}kv \leq v$.  By the above remark, we have that $g^{-1}k\in H$, so $gH = kH$.  We thus see that $a_{kH} = 0$ unless $kH = gH$, in which case $a_{gH} = 1$ (since some term on the right side must be nonzero).  This shows the generating set $\{ (gH)v \mid gH\in G/H \}$ is quasi-independent.
\end{proof}

The above result implies the existence of a $\mathbb{B}$-module $M$ such that $M$ cannot be equipped with the structure of a cyclic $\mathbb{B}[G]$-module for any finite group $G$. 
The following result extends this obstruction to $\mathbb{B}[G]$-modules with a bounded number of generators.

\begin{pro}
Let $M$ be a $\mathbb{B}[G]$-module generated by $n$ elements.  Let $P\subseteq M$ be the poset of join-irreducible elements (which depends only on the $\mathbb{B}$-module structure).  Then $P$ does not contain a chain of length greater than $n$.
\end{pro}
\begin{proof}
From Lemma \ref{lemma: anti-chain lemma}, if $gv \leq kv$ for some $v \in P$ and $g,k \in G$, then $gv=kv$. This implies that any chain in $P$ can intersect an orbit of $P$ in at most one point. And so the length of a chain must be bounded by the number of orbits.
\end{proof}

\bibliography{Vectorbundle}\bibliographystyle{alpha}

\end{document}